\documentclass[11pt,english,a4paper]{amsart}

\usepackage{amssymb}
 \usepackage{amsthm}
 
 \usepackage{ulem}

\usepackage{lineno}

\usepackage{amsmath}

\usepackage{mathrsfs,amsfonts,dsfont} 
\usepackage{mathrsfs}
\usepackage{bbm}
\usepackage{amsfonts}
 \usepackage{dsfont}
 \usepackage{graphicx}
\usepackage{mathrsfs}
\usepackage{wrapfig}
\usepackage{tikz}
\usetikzlibrary{arrows,positioning,shapes.geometric}
\usepackage{amsrefs}
\usepackage{babel}
\usepackage[latin1]{inputenc}

\def\bE{\mathbb E}
\def\bP{\mathbb P}
\def\bR{\mathbb R}
\def\ss{\mathbb S}

\def\cA{{\mathcal A}}

\def\cD{{\mathcal D}}
\def\cR{{\mathcal R}}
\def\G{{\mathcal G}}
\def\PP{(1-\rho^2)}
\def\PPP{(1-\rho^2)^{1/2}}

\def\X {X:=\{X(s),s\in\ss\}}
\def\Y {Y:=\{Y(s),s\in\ss\}}

\def\T{\tau_{s,t}}

\newtheorem{Def}{Definition}
\newtheorem{Th}{Theorem}[section]
\newtheorem{Prop}[Th]{Property}
\newtheorem{Co}[Th]{Corollary}
\newtheorem{Pro}[Th]{Proposition}
\newtheorem{Lem}[Th]{Lemma}

\numberwithin{equation}{section}

\title{Spatial Risk measure for Gaussian processes }
\author{M. Ahmed}
\address{Universit\'e de Lyon, Universit\'e Lyon 1, Institut Camille Jordan ICJ UMR 5208 CNRS, France\\
Department of statistics, University of Mosul, Iraq}
\author{ V. Maume-Deschamps}
\address{Universit\'e de Lyon, Universit\'e Lyon 1, Institut Camille Jordan ICJ UMR 5208 CNRS, France}
\author{P.Ribereau}
\address{Universit\'e de Lyon, Universit\'e Lyon 1, Institut Camille Jordan ICJ UMR 5208 CNRS, France}
\author{C.Vial}
\address{Universit\'e de Lyon, Universit\'e Lyon 1, Institut Camille Jordan ICJ UMR 5208 CNRS, INRIA, Villeurbanne, France}

\begin{document}




\begin{abstract}
{One} of the main characteristic {of climate events is} the spatial dependence.{ In this paper, we study the quantitative behavior of a }spatial risk measure $\cR(\cA,\cD)$ corresponding to a damage function $\cD$ and a region {$\cA\subset\bR^2$}, taking into account the spatial dependence of the underlying process.  This {kind of} risk measure has already been introduced and studied for some max-stable processes in \cite{koch2015spatial}. In this paper, we {consider isotropic Gaussian processes $X=(X(s))_{s\in\bR^2}$ and the excess damage function $\cD_{X\/,u}^+= (X-u)^+$ over a threshold $u\in\bR_+$. We performed a simulation study and a real data study.}
\end{abstract}
\maketitle




\section{Introduction}\label{I:11}

 A heat wave is a prolonged period where maxima temperatures are unusually high with respect to the usual ones. {Most} of the times, these heat waves have also a huge spatial {component}. For example, in 1936, an extremely severe heat wave hits North America. Many states record high temperatures set during this canicule stood until the canicule of 2012. In 2003, a major heat wave hit Europe (cf \cite{HeatWave}, \cite{HeatWave2}), specially France, leaving over 70,000 deaths (around 15,000 only in France). In France, this climatic event was exceptional due to its intensity since some cities registered eight consecutive days with temperature greater than 40$^{\circ}$, but it was also exceptional due to its spatial extent, covering almost all the country. In probability, this means that the underlying spatial process has a strong spatial dependence even at long distance.

 On the other hand, a ``classical'' storm type in the south of France is a {\it cevenol} event. These storms are a particular kind of rainfall, hiting usually the  {\it Cevennes} in France. They are characterized by extremely heavy and localized rainfalls that lead to severe floods. For example, on september 2002, the {\it Gard} department was hitted by an {exceptional} storm. Some locations received more than $700$mm of rain in $24$h. This event caused the death of 23 persons. Another example, on june 15th 2010 Draguignan was severly flooded (cf \cite{Draguignan}, \cite{Draguignan2}), leaving 27 dead and 1 billion Euros of damages. 
 
In both situations, one of the main characteristic of  the event is its spatial dependence: very strong even for large distances for the heat waves and strong at short distances and weak at larger for the cevenol events. When trying to detect the dangerousness of a region using risk measure, {the  notion of spatial dependence must taken into account.}

{Risk measures has been widely studied in literature in the univariate setting (i.e. for random variables). The axiomatic formulation of univariate risk measures has been presented in \cite{artzner1999coherent}. \cite{follmer2014spatial} is concerned with financial products on a network, collections of risk measures indexed by the network are considered.  The extension of definitions of risk measures to spatial processes may take various forms.  As mentioned previously,  one of the main characteristics of these phenomenon is the spatial dependence, for example, in the two previous examples (canicule and cevenol event), after normalisation of the marginals, the risk measure should not be the same since the phenomenons are spatially very different.\\ 
In \cite{keef2009spatial}, the authors proposed to evaluate the risk on a region by a probability $\bP(S<s)$ where $S$ is an integrated damage function. In  \cite{koch2015spatial} or \cite{KOCHErwan2014tools} this idea is developed to define  spatial risk measures taking into account the spatial dependence. In these works, the sensitivity of the risk measures with respect to spatial dependence and space is studied. In the same spirit as \cite{artzner1999coherent}, the author propose a set of axioms that a risk measure in the spatial context should verify. The author focus on max-stable processes. }

{We consider similar spatial risk measures well suited for Gaussian processes. Gaussian processes are relevant for some climate models (temperatures e.g.). As an example, in \cite{leblois2013}, a spatial rainfall generator based on a Gaussian process is proposed. 

Our main contributions concern {the risk measure} based  on the excess damage function $\cD_{X\/,u}^+= (X-u)^+$ over a threshold $u\in\bR_+$. We calculate the risk measure for Gaussian processes; we study its properties with respect to the parameters of the model (with a focus on the dependence parameter). Moreover, we study the axiomatic properties of a class of risk {measures}. }

\ \\
This paper is organized as follows: in {Section} \ref{R:2} we consider quite general spatial risk measures and develop the axiomatic setting of \cite{KOCHErwan2014tools}. Section \ref{R:4} is devoted to the study of the risk measure with damage function $(X-u)^+$ and Gaussian processes.  We propose explicit forms of this risk measure and derive its behavior. We present in Section \ref{R:5} a simulation study in order to evaluate this spatial risk measure. Finally, {we compute the risk measure on the air pollution in Northern Italy model proposed in \cite{bande2006spatio}} in Section \ref{R:6} and  concluding remarks are discussed in Section \ref{R:7}.


\section{Spatial risk measures. }\label{R:2}
Considering a process $X$ we will define a risk measure associated to $X$ on a region {$\cA\subset\bR^2$} of the space. It will be a non negative quantity which represents an average damage or cost due to $X$ on $\cA$. \\
Throughout the paper, $X$ is a spatial process: $X=(X(s)\/, \ s\in \mathbb{S})$ with $\mathbb{S}\subset {\bR^2}$. $\Vert \ \Vert$ denotes the euclidian norm on {$\bR^2$}. 
\subsection{Normalized loss function.}
A damage function $\cD$ represents the {relationship between the aggregate losses} (e.g economic, health) {and} the environmental (climate) indicator (e.g air pollution, temperature levels), {some economic interpretations may be found in  \cite{bosello2007estimating}.}
\begin{Def}\label{SRD-loss} 
\textbf{(Normalized loss function)} Consider a {damage function} $\mathcal{D}:\bR^d\to\bR^+$. For any set  $\cA\in\mathcal{B}(\bR^d)$ the normalized aggregate loss function on $\mathcal{A}$ is
\begin{equation}\label{SRE-loss1}
  L(\cA,\mathcal{D})=\frac{1}{|\cA|}\int_{\cA}\mathcal{D}(s)\quad\mathrm{d}s,
\end{equation}
{where $| \cA|$ stands for the volume (or the Lebesgue measure) of $\cA$.}
\end{Def}
The quantity $\displaystyle\int_{\cA}\mathcal{D}(s)\mathrm{d}s$ represents the aggregated loss over the region $\cA$. Therefore the function $ L(\cA,\mathcal{D})$ is the proportion of loss on $\cA$. {In our context, $\mathcal{D}$ will be a function of the process $X$, denoted $\mathcal{D}_X$.}

More precisely, we will focus on the excess damage function: let $u>0$ be fixed threshold for  $s\in\mathbb{S}$,
\begin{equation}\label{SRE-Dem1}
  \mathcal{D}^+_{X,u}(s)=(X(s)-u)^+.
\end{equation}
For example, when considering air pollutants (like in \cite{bande2006spatio}), $u$ is a regulatory level which is determined  by  experts.
\subsection{Spatial risk measures.}\label{R} 
{As already mentioned, }in spatial contexts the spatial dependency is an important characteristic. {Considering the risk measure as  the expectation of a normalized loss will not take}  into account  the spatial dependency, it is nevertheless useful to quantify the magnitude  of risk with respect to different thresholds $u$. 


{We shall} consider the spatial risk measure composed from two components: the expectation and variance of the normalized loss,
\begin{eqnarray}
\cR(\mathcal{A},\mathcal{D}_X)&=&\{\bE[L(\cA,\mathcal{D}_X)], \mathrm{Var}\big(L(\mathcal{A},\mathcal{D}_X)\big)\} \/, \label{SRE-risk}\\
&=:& \{\cR_0(\mathcal{A},\mathcal{D}_X)\/, \cR_1(\mathcal{A},\mathcal{D}_X)\} \nonumber
\end{eqnarray}
{For stationary processes, the expectation component gives informations on the severity of the phenomenon, while the variance component is impacted by the dependence structure.}\\
Let us remark that
{\begin{equation}\label{SRE-risk1b}
\cR_1(\mathcal{A},\mathcal{D}_X)=\frac{1}{|\cA|^2}\int_{\cA\times\cA}\mathrm{Cov}\big(\cD_X(s),\cD_X(t)\big)\mathrm{d}s\mathrm{d}t.
\end{equation} }
\subsection{Axiomatic properties of spatial risk measures.}\label{R:3}
{In \cite{artzner1999coherent},\cite{krokhmal2007higher},\cite{tsanakas2003risk} and others, axioms and the behavior of univariate risk measures are presented, while \cite{koch2015spatial} provides an axiomatic setting of risk measures in a spatial context.}

In this section we will {present} a set of spatial axiomatic properties describing the behavior of { a real valued spatial risk measure $\cR^*(\cA,\cD)$}. Axioms 1. and 4. below have been introduced in \cite{koch2015spatial}, and studied for some max-stable processes. 
\begin{Def}
Let $\cA\subset\bR^2$ be a region of the space. 
\begin{enumerate}
\item \textbf{Spatial invariance under translation}\\
Let $\cA+v\subset\bR^2$ be the region $\cA$ translated by a vector $v\in\bR^2$. Then for $v\in\bR^2$, {$\cR^*(\cA+v,\cD)=\cR^*(\cA,\cD)$}.

\item \textbf{{Spatial anti-monotoncity}}\\
 Let $\cA_1$,$\cA_2\subset\bR^2$, two regions such that { $|\cA_1|\leq|\cA_2|$}, then {$\cR^*(\cA_2,\cD)\leq\cR^*(\cA_1,\cD)$}.
\item \textbf{{Spatial sub-additivity}}\\
Let $\cA_1$,$\cA_2\subset\bR^2$ be two regions disjointed, then {$\cR^*(\cA_1\cup\cA_2,\cD)\leq\cR^*(\cA_1,\cD)+\cR^*(\cA_2,\cD)$}.
\item \textbf{Spatial super sub-additivity}\\
Let $\cA_1$,$\cA_2\subset\bR^2$ be two regions disjointed, then {$\cR^*(\cA_1\cup\cA_2,\cD)\leq\min_{i=1\/,2}\left[\cR^*(\cA_i,\cD)\right]$}.
\item \textbf{{Spatial homogeneity}}\\
Let $\lambda>0 $ and $\cA\subset \cR^2$ then {$\cR^*(\lambda\cA,\cD)=\lambda^k\cR^*(\cA,\cD)$, that is $\cR^*$ }is homogenous of order $k$, {where $\lambda \cA$ is the set $\{\lambda x, x\in\cA\}$.}

\end{enumerate}
\end{Def}
{Remark that in \cite{koch2015spatial}, the following damage functions are considered for some max-stable processes: $\cD_X(s) = \mathbf{1}_{\{X(s)>u\}}$, $\cD_X(s)= X(s)^\beta$.  The author proves the invariance by translation in this context, he also proves the monotonicity and super sub-additivity in the case where $\cA_1\/, \cA_2$ are either disks or squares.  We shall prove monotonicity for Gaussian processes and the damage function { $\cD_X(s)=\cD^+_{X,u}(s)=(X(s)-u)^+$} in the case where $\cA_1\/, \cA_2$ are either disks or squares (see Section \ref{R:4}).\\ 
\begin{Th}
Let $X$ be a stationary process and $\cD_X$ be a positive damage function of $X$. The risk measure $\cR_1(\cdot\/,\cD_X)$ is invariant by translation and sub-additive.\\
\end{Th}
\begin{proof}
The invariance by translation follows directly from the stationarity. On one other hand, consider $\cA_1$,$\cA_2\subset\bR^2$ two disjointed regions.
\begin{eqnarray*}
 \lefteqn{ {\cR_1}(\cA_1\cup \cA_2\/, \cD_X) =\mathrm{Var}\big(L(\cA_1\cup \cA_2,\cD_X)\big)}\\
 &=&\frac1{(|\cA_1|+|\cA_2|)^2}\left[|\cA_1|^2 \cR_1(\cA_1,\cD_X) +|\cA_2|^2 \cR_1(\cA_2,\cD_X) \right.\\
  &&  +\left. 2  \mathrm{Cov}\left(\int\limits_{\cA_1}\cD_X(s) \/ds\/, \int\limits_{\cA_2}\cD_X(s) \/ds\/\right) \right].\\
  &\leq&  \frac1{(|\cA_1|+|\cA_2|)^2}\left[|\cA_1|^2 \cR_1(\cA_1,\cD_X) +|\cA_2|^2 \cR_1(\cA_2,\cD_X) \right.\\
&&+ \left. 2 |\cA_1||\cA_2| \sqrt{\cR_1(\cA_1,\cD_X)}\sqrt{\cR_1(\cA_2,\cD_X)} \right]\\ 
&&\mbox{by using the Cauchy-Schwarz inequality}\/,\\
&\leq& \cR_1(\cA_1,\cD_X)+\cR_1(\cA_2,\cD_X)\/.
  \end{eqnarray*}
So that we have proved the sub-additivity.  

\end{proof}
}

\section{Risk measures for spatial Gaussian processes. }\label{R:4}
We consider the excess damage function $\cD_{X\/,u}^+ = (X-u)^+$, and $X$ an isotropic standard {spatial Gaussian process on $\ss\subset\bR^2$}, with auto-correlation function $\rho$. {Then for the fixed threshold $u$ and  $\cA\subset\ss$, we have}
\begin{equation}\label{model}
L(\cA,\cD^+_{X,u})=\frac{1}{|\cA|}\int_{\cA}\big(X(s)-u\big)^+\mathrm{d}s \/, \end{equation}  
$$\mbox{and} \ \cR_1(\mathcal{A},\mathcal{D}_{X,u}^+)= \mathrm{Var}(L(\cA,\cD^+_{X,u})) \/.$$
{In what follows, $\varphi$ is the density of the univariate standard normal distribution, $\overline\Phi$ is the survival function of the standard normal distribution, $\ell\big(u,v,w\big)$ is the total probability of a truncated bivariate standard normal distribution with correlation $w${, that is}
\begin{equation}\label{Th-trun}
\ell\big(u,v,w\big)=\frac{1}{2\pi(1-w^2)^{1/2}}\int_{u}^{\infty}\int_{v}^{\infty}e^{\left\{\frac{-1}{2(1-w^2)}[x^2-2w xy+y^2]\right\}}\mathrm{d}x\mathrm{d}y.
\end{equation}

In this Section, we first give explicit forms for the risk measure, then we will study the behavior of $\cR_1(\lambda\mathcal{A},\mathcal{D}_{X,u}^+)$ with respect to $\lambda$.}
\subsection{Explicit forms for {$\cR(\mathcal{A},\mathcal{D}_{X,u}^+)$}}
{We are interested in the explicit calculation of  the expectation and variance of $L(\mathcal{A},\mathcal{D}_{X,u}^+)$. }
\begin{Pro}\label{Pro-risk1}
Consider $\X$ an isotropic standard Gaussian process with auto-correlation function $\rho$. Let  $u\in\bR_+$ be a fixed threshold. We have:{\begin{equation}\label{Pro-exp}
\cR_0(\mathcal{A},\mathcal{D}_{X,u}^+)=\varphi(u)-u\overline{\Phi}(u),
\end{equation}}
and 
{\begin{equation}\label{Th-ex}
\cR_1(\mathcal{A},\mathcal{D}_{X,u}^+)=\frac{1}{|\mathcal{A}|^2}\int_{\mathcal{A}\times\mathcal{A}}\G(\T,u)\quad\mathrm{d}s\mathrm{d}t,
\end{equation}}
with $\tau_{s,\/t} = \Vert s-t\Vert$ {and for any $h\/, s\in\mathbb{S}$}
\begin{equation}\label{Pro-covfun}
{\begin{split}
\G(h,u):=&\mathrm{Cov}\big(\cD^+_{X,u}(s),\cD^+_{X,u}(s+h)\big);\\
\G(h,u)=&\big(\rho(h)+u^2\big)\ell\big(u,u,\rho(h)\big)-2u\varphi(u)\overline{\Phi}\bigg(\frac{u(1-\rho(h))}{(1-\rho^2(h))^{1/2}}\bigg)\\
+&\big(1-\rho^2(h)\big)^{1/2}\varphi\bigg(\frac{u}{(1+\rho(h))^{1/2}}\bigg)^2-\big(\varphi(u)-u\overline{\Phi}(u)\big)^2\/.
\end{split}}
\end{equation}
\end{Pro}
\begin{proof}
Let $X$ be an isotropic standard Gaussian process and $u\in\bR_+$, 
\begin{equation}\label{univar1}
\begin{split}
\bE\big[L(\cA,\cD^+_{X,u})\big]=&\frac{1}{|\cA|}\int_{\cA}\bE\big[(X(s)-u\big)^+\big]\mathrm{d}s\\
=&\frac{1}{|\cA|}\int_{\cA}\bigg[\int_u^{\infty} x{\varphi(x)}\mathrm{d}x-u\int_u^{\infty}{\varphi(x)}\mathrm{d}x\bigg]\mathrm{d}s\\
=&\frac{1}{|\cA|}\int_{\cA}(\varphi(u)-u\overline{\Phi}(u))\mathrm{d}s\\
=&\varphi(u)-u\overline{\Phi}(u).
\end{split}
\end{equation}
{On one other hand,}
$$
\mathrm{Var}\big(L(\mathcal{A},\mathcal{D}_{X,u}^+)\big)=
\frac{1}{|\cA|^2}\int_{\cA\times\cA}\mathrm{Cov}\big(\cD^+_{X,u}(s),\cD^+_{X,u}(t)\big)\mathrm{d}s\mathrm{d}t\/.$$
We calculate $\mathrm{Cov}\big(\cD^+_{X,u}(s),\cD^+_{X,u}(t)\big)$ by using the results from \cite{rosenbaum1961moments} on moments $m_{10}\/, m_{11}$ of truncated bivariate normal distributions. {Let $f_{X_1,X_2}$ be the density function of the random vector $(X_1\/,X_2)$. } 
\begin{equation}\label{univar4}
\begin{split}
\bE\big[\cD^+_{X,u}(s)\cD^+_{X,u}(t)\big]=&\int_u^{\infty}\int_u^{\infty}\big(xy-2ux+u^2\big)f_{X(s),X(t)}\big(x,y\big)\mathrm{d}x\mathrm{d}y\\
=&\ell\big(u,u,\rho(\T)\big)m_{11}-2u\ell\big(u,u,\rho(\T)\big)m_{10}+u^2\ell\big(u,u,\rho(\T)\big),
\end{split}
\end{equation}
{with}
\begin{eqnarray*}
\ell\big(u,v,\rho\big) m_{10}&=&\frac{1}{2\pi\PPP}\int_{u}^{\infty}\int_{v}^{\infty}x\exp\left\{-\frac{1}{2\PP}\big[x^2+2\rho xy+y^2\big]\right\}\mathrm{d}x\mathrm{d}y\/, \\
&=&\varphi(u)\overline{\Phi}\bigg(\frac{v-\rho u}{\PPP}\bigg)+\rho\varphi(v)\overline{\Phi}\bigg(\frac{u-\rho v}{\PPP}\bigg); 
\end{eqnarray*}
and
\begin{eqnarray*}
\ell\big(u,v,\rho\big)m_{11}&=&\frac{1}{2\pi\PPP}\int_{u}^{\infty}\int_{v}^{\infty}xy\exp\left\{-\frac{1}{2\PP}\big[x^2+2\rho xy+y^2\big]\right\}\mathrm{d}x\mathrm{d}y,\\
&=&\rho\ell\big(u,v,\rho\big)+\rho u\varphi(u)\overline{\Phi}\bigg(\frac{v-\rho u}{\PPP}\bigg)+\rho v\varphi(v)\overline{\Phi}\bigg(\frac{u-\rho v}{\PPP}\bigg)\\
&&+\frac{\PPP}{\sqrt{2\pi}}\varphi\bigg(\frac{(u^2-2\rho uv+v^2)^{1/2}}{\PPP}\bigg).
\end{eqnarray*}
{For}  $v=u$, we have, 
$$\ell\big(u,u,\rho\big)m_{10}=(1+\rho)\varphi(u)\overline{\Phi}\bigg(\frac{u(1-\rho)}{\PPP}\bigg)$$
and 
$$\ell\big(u,u,\rho\big)m_{11}=\rho\ell\big(u,u,\rho\big)+2\rho u\varphi(u)\overline{\Phi}\bigg(\frac{u(1-\rho)}{\PPP}\bigg)+\frac{\PPP}{\sqrt{2\pi}}\varphi\bigg(\frac{(2u^2(1-\rho))^{1/2}}{\PPP}\bigg)\/.$$
{Finally we get,}
\begin{multline*}
\begin{split}
\bE\big[\cD^+_{X,u}(s)\cD^+_{X,u}(t)\big]=&\frac{(1-\rho^2(\T))^{1/2}}{\sqrt{2\pi}}\varphi\bigg(\frac{(2u^2(1-\rho(\T)))^{1/2}}{(1-\rho^2(\T))^{1/2}}\bigg)\\
+&2 u\rho(\T)\varphi(u)\overline{\Phi}\bigg(\frac{u(1-\rho(\T))}{(1-\rho^2(\T))^{1/2}}\bigg)\\
+&\rho(\T)\ell\big(u,u,\rho(\T)\big)-2u(1+\rho(\T))\varphi(u)\overline{\Phi}\bigg(\frac{u(1-\rho(\T))}{(1-\rho^2(\T))^{1/2}}\bigg)\\
+&u^2\ell\big(u,u,\rho(\T)\big)\\
\\
=&\ell\big(u,u,\rho(\tau_{s,t})\big)\big(\rho(\tau_{s,t})+u^2\big)-2u\varphi(u)\overline{\Phi}\bigg(\frac{u(1-\rho(\tau_{s,t}))}{(1-\rho^2(\tau_{s,t}))^{1/2}}\bigg)\\
+&\big(1-\rho^2(\tau_{s,t})\big)^{1/2}\varphi\bigg(\frac{u}{(1+\rho(\tau_{s,t}))^{1/2}}\bigg)^2.
\end{split}
\end{multline*}
{The result follows.}
\end{proof}
 
\begin{Co}\label{Co-risk11}
Let $\Y$ be an isotropic Gaussian process with mean $\mu$ and variance $\sigma^2$. {Let $X=\frac{Y-\mu}\sigma$ an isotropic and standard Gaussian process.} The spatial risk measure $\cR(\mathcal{A},\mathcal{D}_{Y,u}^+)$ statisfies
\begin{equation}\label{Th-risk}
\cR(\mathcal A,\mathcal{D}_{Y,u}^+)=\left\{\sigma\bE[L(\mathcal A,\mathcal{D}_{X,u_0}^+)],\sigma^2\mathrm{Var}\big(L(\mathcal A,\mathcal{D}_{X,u_0}^+)\big)\right\},
\end{equation}
with $u_0=(u-\mu)/\sigma$.
\end{Co}
\begin{proof}
{From the definition of $\cD^+_{Y,u}$, we have:}
\begin{equation}\label {Co-change1}
\begin{split}
\bE[L(\mathcal A,\mathcal{D}_{Y,u}^+)]=&\frac{1}{|\cA|}\int_{\cA}\bE(Y(s)-u)^+\mathrm{d}s\\
=&\frac{1}{|\cA|}\int_{\cA}\bE(\mu+\sigma X(s)-u)^+\mathrm{d}s\\
=&\frac{\sigma }{|\cA|}\int_{\cA}\bE\big(X(s)-(\frac{u-\mu}{\sigma })\big)^+\mathrm{d}s\\
=&\sigma\bE[L(\mathcal A,\mathcal{D}_{X,u_0}^+)]\/.
\end{split}
\end{equation}
{On one other hand,}
\begin{gather*}
\mathrm{Var}\big(L(\mathcal A,\mathcal{D}_{Y,u}^+)\big)=\frac{1}{|\cA|^2}\int_{\cA\times\cA}\bE\big[\cD^+_{Y,u}(s)\cD^+_{Y,u}(t)\big]-\bE\big[\cD^+_{Y,u}(s)\big]\bE\big[\cD^+_{Y,u}(t)\big]\mathrm{d}s\mathrm{d}t\\
=\frac{1}{|\cA|^2}\int_{\cA\times\cA}\bE\big[(Y(s)-u)^+(Y(t)-u)^+\big]-\bE\big[(Y(s)-u)^+\big]\bE\big[(Y(t)-u)^+\big]\mathrm{d}s\mathrm{d}t\\
=\frac{1}{|\cA|^2}\int_{\cA\times\cA}\sigma^2\bE\big[(X(s)-u_0)^+(X(t)-u_0)^+\big]-\sigma^2\bE\big[(X(s)-u_0)^+\big]\bE\big[(X(t)-u_0)^+\big]\mathrm{d}s\mathrm{d}t\\
=\frac{\sigma^2}{|\cA|^2}\int_{\cA\times\cA}\bE\big[\cD^+_{X_s,u_0}(s)\cD^+_{X_t,u_0}(s)\big]-\bE\big[\cD^+_{X_s,u_0}(s)\big]\bE\big[\cD^+_{X_t,u_0}(s)\big]\mathrm{d}s.\mathrm{d}t\\
\end{gather*}
Therefore, 
\begin{equation}\label {Co-changevar}
\mathrm{Var}\big(L(\mathcal A,\mathcal{D}_{Y,u}^+)\big)=\sigma^2\mathrm{Var}\big(L(\mathcal A,\mathcal{D}_{X,u_0}^+)\big).
\end{equation}
\end{proof}
{Corollary \ref{Co-risk11} implies that without loss  generality, we may calculate the risk measure  for an isotropic standard Gaussian process, expressions for an isotropic non standard Gaussian process will follow. {Furthermore from these results we can see that $\cR_0(\mathcal A,\mathcal{D}_{Y,u}^+)$ does not depend on the region $\cA$ but only on the characteristics of the underlying Gaussian process. Then in the following study of the risk measure we will focus on the component $\cR_1(\mathcal A,\mathcal{D}_{Y,u}^+)$.}

The following Theorem is useful to compute the risk measure because it reduces to a one dimension integration.}
\begin{Th}\label{Th-shape}
Let $\X$ be an isotropic standard  Gaussian process. If the region $\mathcal{A} $ is either a disk or a square,  the expression $\mathrm{Var}\big(L(\mathcal{A},\mathcal{D}_{X,u}^+)\big)$ {reduces to a one dimensional integration.}\\
\textbf{When $\cA$ is a disk of radius $R$}
\begin{equation}\label{Th-disk}
\mathrm{Var}\big(L(\mathcal{A},\mathcal{D}_{X,u}^+)\big)=\int_{h=0}^{2R}\G(h,u) f_{disk}(h,R)\mathrm{d}h,
\end{equation}
where 
\begin{equation}\label{den2}
f_{disk}(h,R)=\frac{2h}{R^2}\bigg(\frac{2}{\pi}\mathrm{{arccos}}\big(\frac{h}{2R}\big)-\frac{h}{\pi R}\sqrt{1-\frac{h^2}{4R^2}}\bigg),
\end{equation}
{and $\G$ is defined in Equation~(\ref{Pro-covfun}).}

\textbf{When $\cA$ is a  square of side $R$}\\
\begin{equation}\label{Th-square}
\mathrm{Var}\big(L(\mathcal{A},\mathcal{D}_{X,u}^+)\big)=\int_{h=0}^{\sqrt{2}R}\G(h,u)f_{square}(h,R)\mathrm{d}h,
\end{equation}
where, for $h\in[0,R]$
\begin{equation*}
f_{square}(h,R)= \frac{2\pi h}{R^2}-\frac{8h^2}{R^3}+\frac{2h^3}{R^4}
\end{equation*}
and for $h\in[R,\sqrt{2}R]$,
\begin{equation}\label{square}
f_{square}(h,R)=\frac{2h}{R^2}\bigg[-2-b+3\sqrt{b-1}+\frac{b+1}{\sqrt{b-1}}+2\mathrm{arcsin}(\frac{2-b}{b})-\frac{4}{b\sqrt{1-\frac{(2-b)^2}{b^2}}}\bigg],
\end{equation}
where $b=\frac{h^2}{R^2}$.
\end{Th}
\begin{proof} {The strategy of proof is the one adopted in \cite{koch2015spatial} for some max-stable processes.}\\
Let $S$ and $T$ be two independent random variables uniformly distributed on $\cA$. For any function $\gamma$ defined on $\bR^+$, we have  
$$\bE\big[\gamma(||S-T||)\big]=\frac{1}{|\cA|^2}\int_{\cA\times\cA}\gamma(||s-t||)\mathrm{d}s\mathrm{d}t\/.$$
{Using \cite{moltchanov2012distance} if $\cA$ is a  square of side $R$,}
\begin{equation}\label{Th-shape3}
\bE\big[\gamma(||S-T||)\big]=\int_{h=0}^{\sqrt{2}R}\gamma(h)f_{square}(h,R)\mathrm{d}h,
\end{equation}
with $f_{square}$ given by Equation (\ref{square}). {If $\cA$ is a disk of radius $R$ then}
\begin{equation}\label{Th-shape4}
\bE\big[\gamma(||S-T||)\big]=\int_{h=0}^{2R}\gamma(h)f_{disk}(h,R)\mathrm{d}h.
\end{equation}
 Moreover, {by (\ref{Th-ex})}
 $$ \mathrm{Var}\big(L(\mathcal{A},\mathcal{D}_{X,u}^+)\big)=\bE\big[\G(||S-T||,u)\big]\/.$$
 Using { (\ref{Th-shape3}) and (\ref{Th-shape4}) with the function $\gamma(h)=\mathcal{G}(h,u)$ we obtain the result.} 
\end{proof}
In {what follows, we write our results for square regions $\cA$, but the results hold for disks as well.}
\subsection{Behavior  of {$\cR_1(\lambda\cA,\cD^+_{X,u})$} with respect to $\lambda$}
The {following expression of  {$\cR_1(\lambda\cA,\cD^+_{X,u})$}  is a keystone to understand its behavior.}  
\begin{Lem}\label{homo}
Let {$\lambda\geq 0$ and $\cA$ be a square}  of side $R$, then
\begin{equation}\label{Pro-homo4}
\cR_1(\lambda\cA,\cD^+_{X,u})=\int_{h=0}^{\sqrt{2}R}f_{square}(h,R)\G(\lambda h,u)\mathrm{d}h.
\end{equation} 
\end{Lem}
\begin{proof}
{Theorem \ref{Th-shape} gives:}
{\begin{eqnarray*}
\cR_1(\lambda\cA,\cD^+_{X,u})=\mathrm{Var}\big(L(\lambda\mathcal A,\mathcal{D}_{X,u}^+)\big)&=&\int_{h=0}^{\sqrt{2}\lambda R}f_{square}(h,\lambda R)\G(h,u)\mathrm{d}h.\\
&=&\int_{h=0}^{\sqrt{2}R}f_{square}(\lambda h,\lambda R)\G(\lambda h,u)\lambda\mathrm{d}h.
\end{eqnarray*}
}
{Remark} that $f_{square}(\lambda h,\lambda R)=\lambda^{-1}f_{square}(h,R)$. {Thus,}
$$\cR_1(\lambda\cA,\cD^+_{X,u})=\int_{h=0}^{\sqrt{2}R}f_{square}( h,R)\G(\lambda h,u)\mathrm{d}h\/.$$
{The same calculations would give the same result if $\cA$ is a disk of radius $R$ (by replacing $f_{square}$ by $f_{disk}$).}
\end{proof}
Lemma \ref{homo} gives the following two results on the behavior of the mapping {$\lambda\mapsto\cR_1(\lambda\cA,\cD^+_{X,u})$.}
\begin{Co}\label{Prop-dec}
Let $X$ be an isotropic standard Gaussian process {on $\ss\subset\bR^2$} with auto-correlation function $\rho$. {Let {$\cA\subset\ss$} be either a disk or a square}. The mapping {$\lambda\mapsto\cR_1(\lambda\cA,\cD^+_{X,u})$} is non-increasing if and only if $h\mapsto\rho(h)$, $h>0$ is non-increasing {and non-negative} . 
\end{Co}
\begin{proof}
{It suffices to remark that by its definition, for any $h>0$, the function $\lambda \mapsto \G( \lambda h\/,u)$ is non-increasing provided the auto-correlation function is {non-negative} and   non-increasing.}
\end{proof}

\begin{Co}\label{Th-AIG}
Let $\X$ be an isotropic  standard Gaussian process with auto-correlation function satisfying {decreasing to $0$ as $h$ goes to infinity.} Then, for $\cA$ either a disk or a square, {we have}
{\begin{equation}\label{Th-AIR}
\lim_{\lambda\to\infty}\cR_1(\lambda\cA,\cD^+_{X,u})=0.
\end{equation}}
\end{Co}
\begin{proof}
{Let $ \cA$ be a square of side $ R$,}
\begin{equation}\label{Th-AIG1}
\cR_1(\lambda\cA,\cD^+_{X,u})=\int_{h=0}^{\sqrt{2}R}f_{square}(h,R)\G(\lambda h,u)\mathrm{d}h\/,
\end{equation}
the monotonic convergence theorem gives:
\begin{equation}\label{Th-AIG2}
\lim_{\lambda\to\infty} \cR_1(\lambda\cA,\cD^+_{X,u})=\int_{h=0}^{\sqrt{2}R}f_{square}(h,R)\lim_{\lambda\to\infty}\G(\lambda h,u)\mathrm{d}h.
\end{equation}
{Since $\rho(h)$ goes to $0$ as $h$ goes to infinity  
$$\lim_{\lambda\to\infty}\G(\lambda h,u)=u^2\ell\big(u,u,0\big)-u^2\overline{\Phi}^2(u)\/$$
and the result follows.}
\end{proof}
{We finish this section with the remark that Lemma \ref{homo} implies the anti-monotonicity for regions $\cA_1$, $\cA_2$ which are either a disk or a square. 
\begin{Prop}\label{anti}
Let $\X$ be an isotropic standard Gaussian process with non-negative and non-increasing auto-correlation function, let $\cA_1$, $\cA_2$ be either squares or disks such that $|\cA_1|\leq |\cA_2|$ then 
{$$\cR_1(\lambda\cA_2,\cD^+_{X,u})\leq \cR_1(\lambda\cA_1,\cD^+_{X,u})\/.$$}
\end{Prop}
\begin{proof}
Let us do the proof in the square case. By invariance by translation, we may assume $\cA_1=\lambda \cA_2$ for some $\lambda\leq1$. Equation (\ref{Pro-homo4}) gives the result. 
\end{proof}
We now perform some simulation study for various shapes of auto-correlation functions.
}
 \section{Simulation study}\label{R:5}
{In this section, we study}   the behavior {of} the proposed spatial risk measure $\cR(\cA,\cD^+_{X,u})$,  {through some simulations}. 
 \subsection{Analysis of $\mathcal{G}(h,u)$ and $\cR_1(\lambda\cA,\cD^+_{X,u})$ }\label{corre}
{We begin this simulation section with the study of the covariance damage function $\mathcal{G}$ which plays a central role in the behavior of $\cR(\cA,\cD^+_{X,u})$. Following \cite{abrahamsen1997review}, we consider five Gaussian models {depending on the choice of the correlation structure, more precisely,}
\begin{enumerate}
  \item Spherical correlation function:
        \begin{equation*}
          \rho^{sph}_{\theta}(h)=\bigg[1-1.5\bigg(\frac{h}{\theta}\bigg)+0.5\bigg(\frac{h}{\theta}\bigg)^3\bigg]\mathds{1}_{\{h>\theta\}}.
        \end{equation*}
  \item Cubic correlation function :
        \begin{equation*}
          \rho^{cub}_{\theta}(h)=\bigg[1-7\bigg(\frac{h}{\theta}\bigg)+\frac{35}{2}\bigg(\frac{h}{\theta}\bigg)^2-\frac{7}{2}\bigg(\frac{h}{\theta}\bigg)^5+\frac{3}{5}\bigg(\frac{h}{\theta}\bigg)^7\bigg]\mathds{1}_{\{h>\theta\}}.
        \end{equation*}
  \item { Exponential correlation functions:
        \begin{equation*}
          \rho^{exp}_{\theta}(h)=\exp\big[-\frac{h}{\theta}\big],
        \end{equation*}}
 \item { Gaussian correlation functions:
        \begin{equation*}
           \rho^{gau}_{\theta}(h)=\exp\big[-\big(\frac{h}{\theta}\big)^2\big];
        \end{equation*}}
 \item {Mat\'ern} correlation function: 
\begin{equation*}
 \rho^{mat}(h)=\frac{1}{\Gamma(\kappa)2^{\kappa-1}}( h/\theta)^{\kappa}K_\kappa(h/\theta)\/,
\end{equation*}
\end{enumerate}
where $\Gamma$ is the gamma function, $K_{\kappa}$ is the modified Bessel function of second kind and order $\kappa>0$, $\kappa$ is a smoothness parameter and $\theta$ is a scaling parameter.

{In order to emphasize the dependence of the damage covariance function $\mathcal{G}$ to the correlation parameter we will denote it by $\mathcal{G}_{\theta}(h,u)$ for any triplet $(h,u,\theta)$.}

\begin{figure}[h]
\centering
\fboxrule=1pt 
\fbox{\includegraphics[width=13cm]{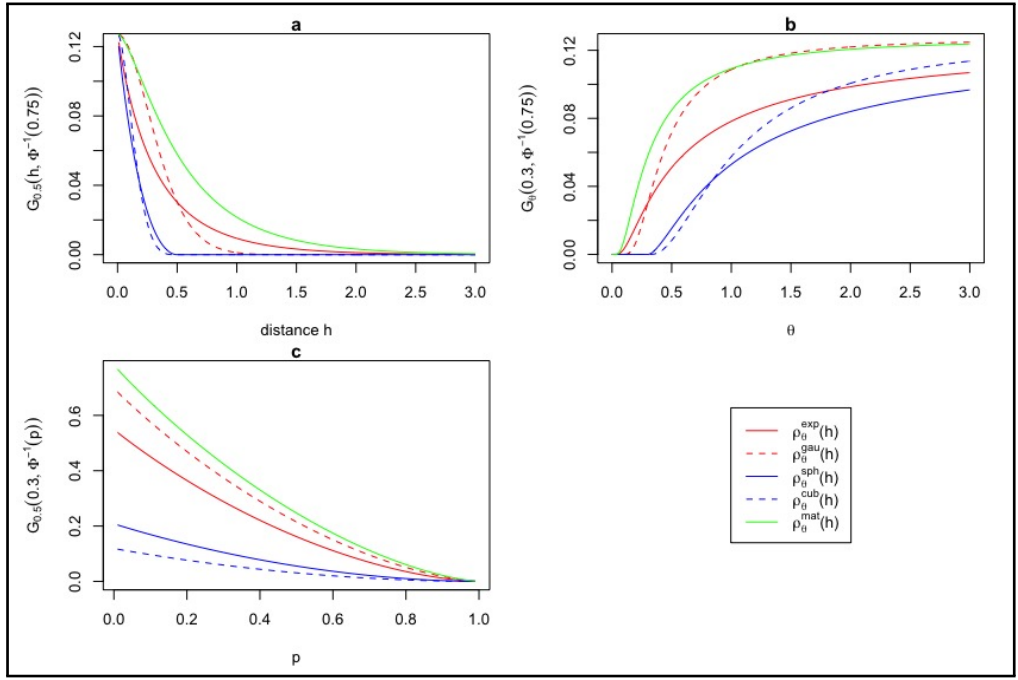}}
\caption{ {Behavior of $\mathcal{G}_{\theta}(h,u)$ with respect to the threshold $u$, the correlation length $\theta$ and  the distance $h$. Five non-negative correlation functions (exponential, Gaussian, spherical, cubic {and Mat\'ern with $\kappa=1$})  have been examined. The graphs (a), (b) and (c) show the behavior of $\mathcal{G}_\cdot(\cdot,\cdot)$ with respect to: (a)  the distance $h$, when $u=\Phi^{-1}(0.75)$ and $\theta=0.50$; (b) the correlation length $\theta$, when $u=\Phi^{-1}(0.75)$ and $h=0.30$; (c) the threshold $u=\Phi^{-1}(p),p\in[0,1]$, when $\theta=0.50$ and $h=0.30$.}}
 \label{GauF1}
\end{figure}
{ Figure \ref{GauF1}.(a) shows the behavior of the spatial covariance between two damage functions $\cD_{X,u}^+(\cdot)$ and $\cD_{X,u}^+(\cdot+h)$ with respect to the distance $h$, when the correlation length is set to $\theta=0.50$ and the threshold to $u=\Phi^{-1}(0.75)$,  where $\Phi^{-1}$ is  the quantile function of the standard normal distribution. It shows that $\mathcal{G}_{\theta}(h,u)$ tends  to 0 as $h$ tends to infinity { with different decreasing speed.}} {This is obviously the expected behavior, because the process $(\cD_{X,u}^+(s), s\in\mathbb{S})$ is (spatially) asymptotically independent. Whereas, for spherical and cubic correlation functions, $\mathcal{G}_\theta(h,u)=0$ as soon as $h>\theta$, which means that the process $(\cD_{X,u}^+(s),s \in\mathbb{S})$ is  $\theta$-independent (independent at distance larger that $\theta$).  
}

{ In order to study the behavior of the damage covariance function with respect to $\theta$, we set the threshold $u=\Phi^{-1}(0.75)$ and the distance $h=0.30$. In Figure \ref{GauF1}.(b) we remark that  $\mathcal{G}_\theta(h,u)$ is increasing {with $\theta$.} 
 \\}
{ Finally, {we study} the behavior of the damage covariance function with respect to the threshold $u=\Phi^{-1}(p),p\in[0,1]$. We set $\theta=0.50$ and $h=0.30$. {Remark} (see Figure~\ref{GauF1}.(c)) {that  even if $h$ is small, $\mathcal{G}_\theta(h,\Phi^{-1}(p))$  goes to zero as $p$ goes to 1, so that it will be difficult to approximate correctly  the covariance damage function when $u$ is large.}
 }

\begin{figure}[h]
\centering
\fboxrule=1pt 
\fbox{\includegraphics[width=13cm]{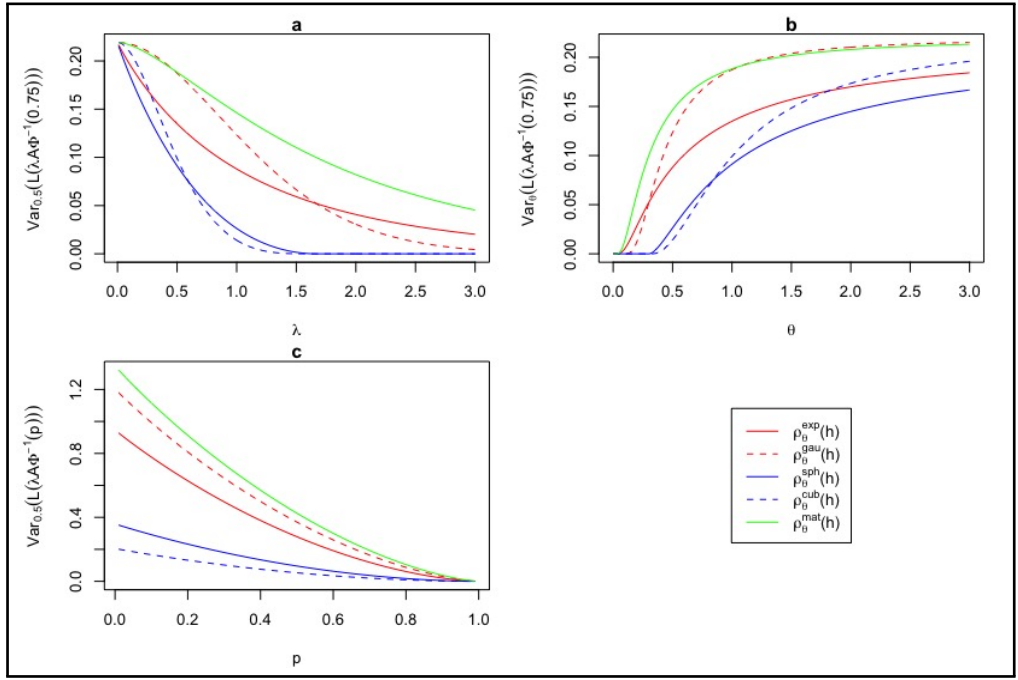}}
\caption{The behavior of {$\cR_1(\lambda\cA,\cD^+_{X,u})$} {for}  $\cA=[0,1]^2$. Exponential, Gaussian, spherical, cubic and Mat\'ern with $\kappa=1$ non-negative correlation functions. The graphs (a), (b) and (c) show the  behavior of {$\cR_1(\lambda\cA,\cD^+_{X,u})$} for a fixed $h=0.30$ with respect to : (a)  $\lambda$, when  $u=\Phi^{-1}(0.75)$ and $\theta=0.50$; (b)  $\theta$, when $u=\Phi^{-1}(0.75)$ and $\lambda=1$; (c)  $u=\Phi^{-1}(p),p\in[0,1]$, when $\lambda=1$ and $\theta=0.50$. \label{GauF2}}
\end{figure}

{Figure~\ref{GauF2}. focusses on the behavior of {$\cR_1(\lambda\cA,\cD^+_{X,u})$} with respect to  $(\lambda,u,\theta)$, when $\cA$ is a square of side $R=1$. In order to see the influence of the homothety rate $\lambda$, we set  $u=\Phi^{-1}(0.75)$ and $\theta=0.50$. 
}

{To tackle the behavior with respect to $\theta$ we choose  $\lambda=1$ and $u=\Phi^{-1}(0.75)$. 
}

{To study the behavior of the variance with respect to the threshold $u=\Phi^{-1}(p),p\in[0,1]$, we set $\lambda=1$ and $\theta=0.50$. 
} 

 \subsection{Numerical computation}
 We generated {isotropic} standard spatial Gaussian {processes} $X$ {on $\ss=\bR^2$} with different non-negative correlation functions (exponential, Gaussian, spherical, cubic and Mat\'ern with $\kappa=1$) for $\theta=0.5$.  The process $X$ is simulated on a $(15\times 15)$ {irregular} grid with $n=125$ locations over {$\cA=[0,1]^2$}.
We set the threshold $u=\Phi^{-1}(p)$, for $p:=\{0.75,0.85,0.95\}$.\\
\ \\
{This section is devoted to a numerical study of the computation of {$\cR_1(\lambda\cA,\cD^+_{X,u})$}, where $\cA=[0,1]^2$. We compare the computation of {$\cR_1(\lambda\cA,\cD^+_{X,u})$} by the one dimensional integration using (\ref{Th-ex}) with the {intuitive} Monte-Carlo computation (M1).  The (M1) computation is obtained by  generating a $m=1000$ sample of $X$ on the grid. That is,}

\begin{equation}
L_j(\mathcal{A},\mathcal{D}_{X,u}^+)=\frac{1}{|\mathcal{A}|}\bigg[\frac{1}{n-1}\bigg]^2\sum_{i=1}^{n-1}{(X(s_{ij})-u)^+}\quad j=1,...,m.
\end{equation}
\begin{equation}
{\cR_0^{M1}(\lambda\cA,\cD^+_{X,u})}=\frac{1}{m}\sum_{j=1}^{m}L_j(\mathcal{A},\mathcal{D}_{X,u}^+)
\end{equation}
and 
\begin{equation}\label{simulation}
{\cR_1^{M1}(\lambda\cA,\cD^+_{X,u})}=\mathrm{Var}(L_j(\mathcal{A},\mathcal{D}_{X,u}^+)).
\end{equation}
{Boxplots in Figure \ref{GauF3} represent  the relative errors over $100$ (M1) computations  with respect to the one dimensional integration. } 
\begin{figure}[h]
\centering
\fboxrule=1pt 
\fbox{\includegraphics[width=13cm]{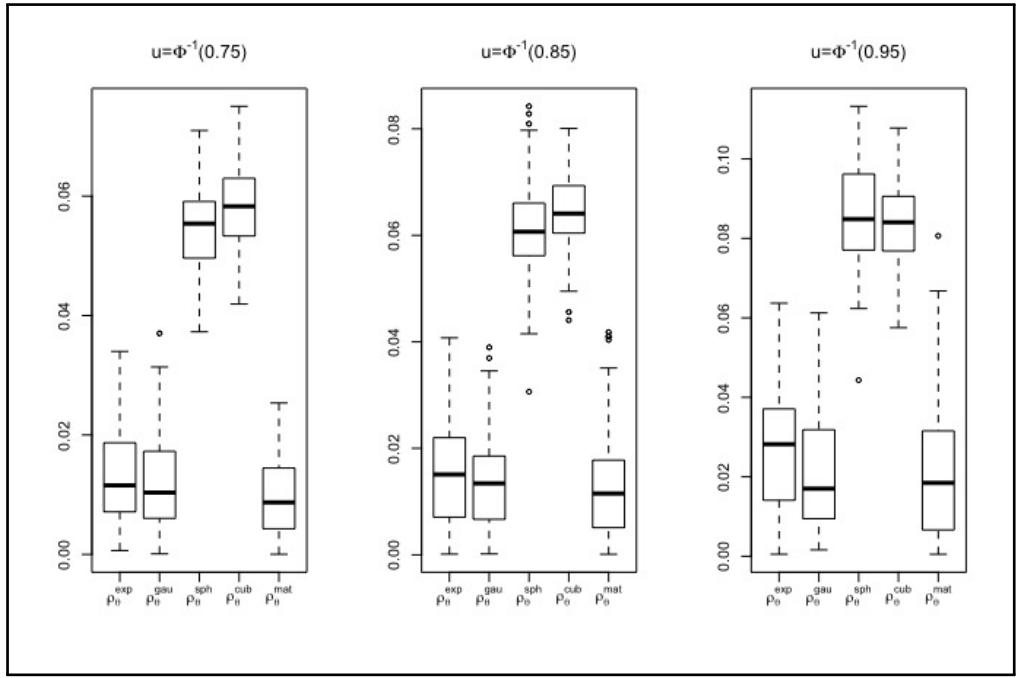}}
\caption{ The boxplots represent the relative errors of $\cR_1(\lambda\cA,\cD^+_{X,u})$ between {the one dimensional integration computation and the M1  method for different thresholds} $u=\Phi^{-1}(p)$, $p:=\{0.75,0.85,0.95\}$ and  five correlation functions (exponential, Gaussian, spherical, { cubic and Mat\'ern with $\kappa=1$) for correlation length $\theta=0.5$ over  $\cA=[0,1]^2$}.\label{GauF3}}
\end{figure}
{Because exponential, Gaussian and matern correlation models} {have relatively simple forms, the relative errors are expected to be smaller} compared to spherical and cubic ones.   {For cubic and spherical models, the discontinuity at $h=\theta$ induces more instability in the simulations.}     
\section{{Piemonte case study}}\label{R:6}
{We terminate this paper with the computation of the risk measure $\cR_1(\lambda\cA,\cD^+_{X,u})$ on pollution in Piemonte data. The  air pollution is measured by the concentration in $PM_{10}$, particulate matter with an aerodynamic diameter less than $10\mu m$. The observed values of $PM_{10}$  are frequently larger than the legal level fixed by the European directive $2008/50/EC$ (see \cite{cameletti2013spatio} for details). \\
The data  has been fitted and analyzed in  \cite{bande2006spatio}. The data contains the daily concentration of $PM_{10}$ during the winter season 2005-March 2006. {The authors} considered 24 monitoring stations for estimating the parameters of this model and 10 stations for validation.}\\
\\
\begin{figure}[h]
\centering
\fboxrule=0pt 
\fbox{\includegraphics[width=10cm, height=11cm]{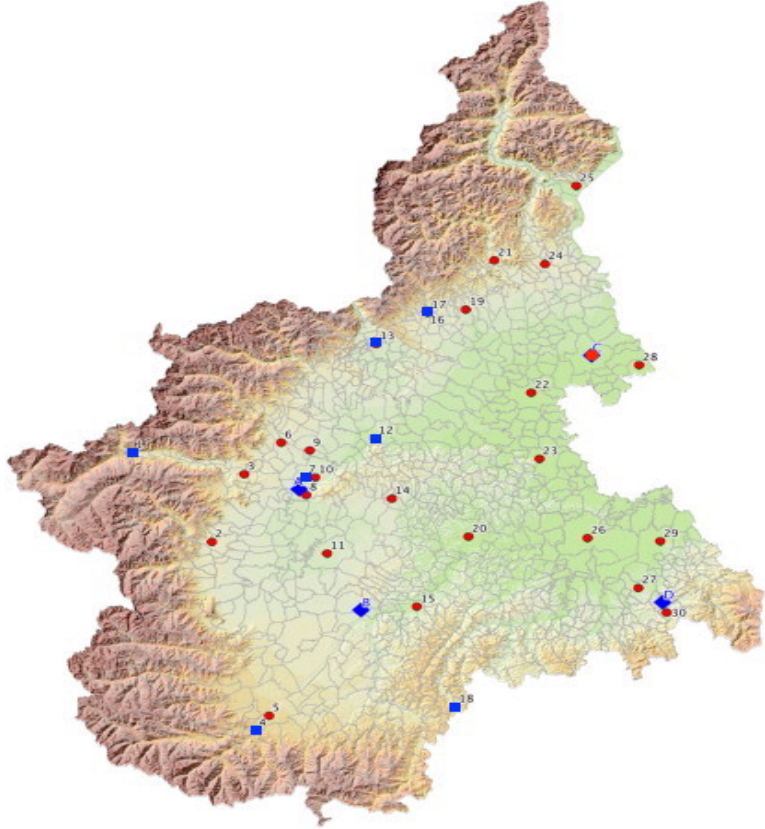}}\
\caption{\cite{bande2006spatio}. Locations of the 24 $PM_{10}$ monitoring sites (red dots) and 10 validation stations (blue squares) in northern Italy between Alps and Appenises (Piemonte region).}
\end{figure}
{The log of $PM_{10}$ has been fitted on an isotropic Gaussian process with Mat\'ern auto-correlation function. In what follows, $Y=\log PM_{10}$. Following the parameter estimation {(see \cite{bande2006spatio})}, we will use} { $\kappa=1$ and $\theta=100$. The estimation of the marginal parameters leads us to use $\mu=3.69$ and $\sigma^2=1.2762$.} 
\\
{We use the above parameters to compute the risk measure 
{$$\left(\cR_0(\lambda\cA,\cD^+_{Y,\log u})\/,  \cR_1(\lambda\cA,\cD^+_{Y\/,\log u})\right)\/,$$} 
with $\cA$ a square of side $10$km and $u$ the legal level, {i.e. $u=50$}. We use Corollary \ref {Co-risk11}, let $Y_0=\frac{Y-\mu}\sigma$ and $ u_0=(\log(50)-3.96)/\sqrt{1.2762}=0.1965$, we have}
\begin{equation*}
\begin{split}
{\cR_0(\lambda\cA,\cD^+_{Y,\log u})}=&\sqrt{1.2762}\big(\varphi(0.1965)-0.1965\overline{\Phi}(0.1965)\big)\\
=&0.3483621
\end{split}
\end{equation*}
and 
\begin{equation*}
\begin{split}
{ \cR_1(\lambda\cA,\cD^+_{Y,\log u})}=&1.2762\int_{h=0}^{14.15}f_s(h,1)\mathcal{G}(h,0.1956)\mathrm{d}h.\\
  =&0.4119461.
 \end{split}
\end{equation*}
{The random variable {$L(\cA\/,\cD^+_{Y,\log u})$} is the average over the square $\cA$ of the values of $Y$ that exceed the legal threshold $\log u$. This is a quantity of interest for health public policies. Our study shows that the standard deviation of {$L(\cA\/,\cD_{Y,\log u}^+)$} is large with respect to its expectation. This means that the dependence structure of the underlying process highly impacts the random variable {$L(\cA\/,\cD_{Y,\log u}^+)$.}
\section{Conclusion}\label{R:7}
We have proposed a spatial risk measure $\cR(\cA,\cD_{X,u}^+)$. {It} takes into account the spatial dependence {over} a region. {We showed that some proposed axioms are valid for any stationary processes. Properties such as anti-monotonicity is verified for isotropic Gaussian processes and $\cA$ a disk or a square (the same result holds for some max-stable processes, see \cite{koch2015spatial}). A simulation study emphasized the behavior of the risk measure with respect to the various parameters. Finally, the computation on pollution data showed the interest of using the variance of  {$L(\cA\/,\cD_{Y,\log u}^+)$} as a spatial risk measure in  concrete cases. }
\\
\ \\
\underline{Acknowledgements:} This work was supported by the LABEX MILYON (ANR-10-LABX-0070) of Université de Lyon, within the program "Investissements d'Avenir" (ANR-11-IDEX-0007) operated by the French National Research Agency (ANR).






\bibliographystyle{apalike}
\bibliography{sample.bib}

\begin{bibdiv}
\begin{biblist}

\bib{abrahamsen1997review}{book}{
      author={Abrahamsen, P.},
       title={A review of gaussian random fields and correlation functions},
   publisher={Norsk Regnesentral/Norwegian Computing Center},
        date={1997},
}

\bib{artzner1999coherent}{article}{
      author={Artzner, P.},
      author={Delbaen, F.},
      author={Eber, J.-M.},
      author={Heath, D.},
       title={Coherent measures of risk},
        date={1999},
     journal={Mathematical finance},
      volume={9},
      number={3},
       pages={203\ndash 228},
}

\bib{bande2006spatio}{article}{
      author={Bande, Stefano},
      author={Ignaccolo, Rosaria},
      author={Nicolis, Orietta},
       title={Spatio-temporal modelling for pm10 in piemonte},
        date={2006},
     journal={Atti della XLIII Riunione Scientifica della SIS},
       pages={87\ndash 90},
}

\bib{bosello2007estimating}{article}{
      author={Bosello, F.},
      author={Roson, R.},
       title={Estimating a climate change damage function through general
  equilibrium modeling},
        date={2007},
     journal={Department of Economics, University Ca'Foscari of Venice Research
  Paper},
      number={08-07},
}

\bib{cameletti2013spatio}{article}{
      author={Cameletti, M.},
      author={Lindgren, F.},
      author={Simpson, D.},
      author={Rue, H.},
       title={Spatio-temporal modeling of particulate matter concentration
  through the spde approach},
        date={2013},
     journal={AStA Advances in Statistical Analysis},
      volume={97},
      number={2},
       pages={109\ndash 131},
}

\bib{HeatWave}{article}{
      author={Chase, T.~N.},
      author={Wolter, K.},
      author={Sr, R. A.~Pielke},
      author={Rasool, I.},
       title={Was the 2003 european summer heat wave unusual in a global
  context?},
        date={2006},
     journal={Geophysical Research Letters},
      volume={33},
}

\bib{follmer2014spatial}{article}{
      author={F{\"o}llmer, H.},
       title={Spatial risk measures and their local specification: The locally
  law-invariant case},
        date={2014},
     journal={Statistics \& Risk Modeling},
      volume={31},
      number={1},
       pages={79\ndash 101},
}

\bib{HeatWave2}{article}{
      author={G.-Herrera, R.},
      author={Diaz, J.},
      author={Trigo, J.~M.},
      author={Luterbacher, J.},
      author={M.Fisher, E.},
       title={A review of the european summer heat wave of 2003},
        date={2010},
     journal={Critical Reviews in Environmental Science and Technology},
      volume={40},
       pages={267\ndash 306},
}

\bib{keef2009spatial}{article}{
      author={Keef, C.},
      author={Tawn, J.},
      author={Svensson, C.},
       title={Spatial risk assessment for extreme river flows},
        date={2009},
     journal={Journal of the Royal Statistical Society: Series C (Applied
  Statistics)},
      volume={58},
      number={5},
       pages={601\ndash 618},
}

\bib{KOCHErwan2014tools}{thesis}{
      author={Koch, E.},
       title={Tools and models for the study of some spatial and network
  risks:application to climate extremes and contagion in france},
        type={Ph.D. Thesis},
     address={ISFA, University of Claude Bernard Lyon1},
        date={2014},
}

\bib{koch2015spatial}{article}{
      author={Koch, E.},
       title={Spatial risk measures and applications to max-stable processes},
        date={2015},
     journal={To appear in Extremes},
}

\bib{krokhmal2007higher}{article}{
      author={Krokhmal, P.~A.},
       title={Higher moment coherent risk measures},
        date={2007},
     journal={Quantitative Finance},
      volume={7},
      number={4},
       pages={373\ndash 387},
}

\bib{leblois2013}{article}{
      author={Leblois, E.},
      author={Creutin, J.D.},
       title={Space-time simulation of intermittent rainfall with prescribed
  advection field: Adaptation of the turning band method},
        date={2013},
     journal={Water Resources Research},
      volume={49},
       pages={3375\ndash 3387},
}

\bib{Draguignan2}{article}{
      author={Martin, C.},
       title={Les inondations du 15 juin 2010 dans le centre var : r\'eflexion
  sur un \'episode exceptionnel},
        date={2010},
     journal={Etudes de G\'eographie Physique},
      volume={XXXVII},
       pages={41\ndash 76},
}

\bib{moltchanov2012distance}{article}{
      author={Moltchanov, D.},
       title={Distance distributions in random networks},
        date={2012},
     journal={Ad Hoc Networks},
      volume={10},
      number={6},
       pages={1146\ndash 1166},
}

\bib{Draguignan}{inproceedings}{
      author={Payrastre, O.},
      author={Gaume, E.},
      author={Javelle, P.},
      author={Janet, B.},
      author={Fourmigu\'e, P.},
      author={Lefort, P.},
      author={Martin, A.},
      author={Boudevillain, B.},
      author={Brunet, P.},
      author={Delrieu, G.},
      author={Marchi, L.},
      author={Aubert, Y.},
      author={Dautrey, E.},
      author={Durand, L.},
      author={Lang, J.},
      author={Boissier, L.},
      author={Douvinet, J.},
      author={Martin, C.},
      author={Ruin, I.},
      author={of~HYMEX, TTO2D~Team},
       title={Analyse hydrologique de la catastrophe du 15 juin 2010 dans la
  r\'egion de draguignan (var, france)},
        date={2012},
   booktitle={Congr\'es shf : Ev\'enements extr\^emes fluviaux et maritimes,
  paris},
}

\bib{rosenbaum1961moments}{article}{
      author={Rosenbaum, S.},
       title={Moments of a truncated bivariate normal distribution},
        date={1961},
     journal={Journal of the Royal Statistical Society. Series B
  (Methodological)},
       pages={405\ndash 408},
}

\bib{tsanakas2003risk}{article}{
      author={Tsanakas, A.},
      author={Desli, E.},
       title={Risk measures and theories of choice},
        date={2003},
     journal={British Actuarial Journal},
      volume={9},
      number={04},
       pages={959\ndash 991},
}

\end{biblist}
\end{bibdiv}

\end{document}